\documentclass[12pt,reqno,draft]{amsart}
\usepackage{enumerate, latexsym, amsmath, amsfonts, amssymb, amsthm, graphicx, color}
\def\pmod #1{\ ({\rm{mod}}\ #1)}
\def\Z{\Bbb Z}
\def\N{\Bbb N}

\def\Q{\Bbb Q}

\def\r{\right}
\def\bg{\bigg}
\def\({\bg(}
\def\){\bg)}

\def\gen{{\rm gen}}
\def\q{{\rm q}}
\def\b{{\rm b}}

\def\spn{{\rm spn}}

\def\ve{\varepsilon}

\def\Ack{\medskip\noindent {\bf Acknowledgments}}
\theoremstyle{plain}
\newtheorem{theorem}{Theorem}

\newtheorem{lemma}{Lemma}

\theoremstyle{definition}

\theoremstyle{remark}

\makeatletter
\@namedef{subjclassname@2010}{%
  \textup{2010} Mathematics Subject Classification}
\makeatother
 \vspace{4mm}

\begin{document}
 \baselineskip=17pt
\hbox{}
\medskip
\title[Almost universal ternary sums of pentagonal numbers]
{Almost universal ternary sums of pentagonal numbers}
\date{}
\author[Hai-Liang Wu and Li-Yuan Wang] {Hai-Liang Wu and Li-Yuan Wang}

\thanks{
\newline \indent Supported by the National Natural Science
Foundation of China (Grant No. 11971222).}

\address {(Hai-Liang Wu) School of Science, Nanjing University of Posts and Telecommunications, Nanjing 210023, People's Republic of China}
\email{\tt whl.math@smail.nju.edu.cn}

\address
{(Li-Yuan Wang) School of Physical and Mathematical Sciences, Nanjing Tech
University, Nanjing 211816, People's Republic of China}
\email {wly@smail.nju.edu.cn}

\begin{abstract}
For each integer $x$, the $x$-th generalized pentagonal number is denoted by $P_5(x)=(3x^2-x)/2$. Given odd positive integers $a,b,c$ and non-negative integers $r,s$, we employ the theory of ternary quadratic forms to determine when the sum $aP_5(x)+2^rbP_5(y)+2^scP_5(z)$ represents all but finitely many positive integers.
\end{abstract}

\maketitle
\section{Introduction}
\setcounter{lemma}{0}
\setcounter{theorem}{0}
\setcounter{corollary}{0}
\setcounter{remark}{0}
\setcounter{equation}{0}
\setcounter{conjecture}{0}
Given an integral positive definite ternary quadratic form $f(X_1,X_2,X_3)$, it is well known that
there are infinitely many positive integers that can not be represented by $f$ over $\Z$. For example, the famous Legendre three-square theorem shows that every positive integer in the set $\{4^k(8l+7): k,l=0,1,\cdots\}$ can not be represented as the sum of three squares. Hence many mathematicians started to investigate inhomogeneous quadratic polynomials. In particular, quadratic polynomials involving polygonal numbers have many interesting properties. For each integer $m\ge3$ and integer $x$, the $x$-th generalized $m$-gonal number is denoted by $P_m(x)=((m-2)x^2-(m-4)x)/2$. Fermat conjectured that every positive integer $n$ can be written as the sum of $m$ generalized $m$-gonal numbers, i.e., $n=P_m(x_1)+\cdots+P_m(x_m)$ always has integral solutions. The case $m=3$ was confirmed by Gauss. The case $m=4$ was proved by Lagrange and this is known as the Lagrange four squares theorem. The case $m\ge5$ was totally resolved by Cauchy (cf. \cite[pp. 3--35]{Na}).

In this line, Liouville studied the quadratic polynomial $aT_x + bT_y + cT_z$, where $T_x=x(x+1)/2$ is the $x$-th generalized triangular number. And in 1862 he succeeded in determining all triples of positive integers $(a,b,c)$ for which the sums $aT_x + bT_y + cT_z$ represent all positive integers.

In 2010, Kane and Sun \cite{KS} investigated the mixed sums of squares and triangular numbers that can represent all but finitely many positive integers. This work was later completed by Chan and Oh \cite{CO} and Chan and Haensch \cite{WKCANNA}. In 2015, Sun \cite{S15} studied the ternary sums of the generalized pentagonal numbers and showed that there are at most $11$ triples $(a_1,a_2,a_3)\in\N^3$ such that $a_1P_5(x)+a_2P_5(y)+a_3P_5(z)$ represents all positive integers. All of these candidates were later confirmed by F. Ge and Sun \cite{GS}, and B.-K. Oh \cite{OH5}. Recently, Ju \cite{Ju} showed that given a quadratic polynomial $a_1P_5(x_1)+\cdots+a_kP_5(x_k)$ with $a_1,\cdots,a_k$ positive integers, it can represent all positive integers if and only if it can represent $1, 3, 8, 9, 11, 18, 19, 25, 27, 43, 98$, and $109$.

Motivated by the above work, in this paper we focus on the ternary sums
$$\mathcal{F}(x,y,z)=aP_5(x)+2^rbP_5(y)+2^scP_5(z),$$
where $a,b,c$ are positive odd integers and $r,s$ are nonnegative integers. For convenience, a quadratic polynomial is called almost universal if it can represent all but finitely many positive integers over $\Z$. For the recent work on almost universal quadratic polynomials, readers may refer to \cite{CO,WKCANNA,A1,A2,ANNAKEN,KS,MS,S15,S17,X}.

It is clear that $n$ can be represented by $\mathcal{F}$ if and only if there are $x,y,z\in\Z$ such that
$$24n+a+2^rb+2^sc=a(6x-1)^2+2^rb(6y-1)^2+2^sc(6z-1)^2.$$
To solve our problems, we use the language of quadratic space and lattice, which we will give a brief introduction in Section 2. Let $(V,\b,\q)$ be a quadratic space over $\Q$ with symmetric bilinear map $\b$ and associated quadratic map $\q$.
Let $\{{\bf e_1},{\bf e_2}, {\bf e_3}\}$ be an orthogonal basis of $V$ satisfying
$$\q({\bf e_1})=a,\ \q({\bf e_2})=2^rb,\ {\rm and}\ \q({\bf e_3})=2^sc.$$
We let
$$L=\Z{\bf e_1}+\Z{\bf e_2}+\Z{\bf e_3},$$
and let
$${\bf v}=-({\bf e_1}+{\bf e_2}+{\bf e_3}).$$
Then one can verify that $n$ can be represented by $\mathcal{F}$ if and only if there is a vector ${\bf w}\in L$ such that $24n+a+2^rb+2^sc=\q({\bf v}+6{\bf w})$. Let $M:=\Z{\bf v}+6L$, and for each prime $p$, we let $M_p:=M\otimes_{\Z}\Z_p$, where $\Z_p$ denotes the ring of all $p$-adic integers.
We will see in Section 2 that the lattice $M$ has close relations with the global representation.

In the remaining part of this paper, we always assume that the following condition $(*)$ is satisfied:
$$(a,b,c)=1,\ 0\le r\le s,\ {\rm and}\ (6, a+2^rb+2^sc)=1.\ (*)$$
For each prime $p$ and any $p$-adic integer $x$, the symbol $\nu_p(x)$ denotes the $p$-adic order of $x$. Also, for any positive integer $y$, the square-free part of $y$ is denoted by $\mathfrak{sf}(y)$. Moreover, we say that $\mathcal{F}$ has no local obstruction if $\{\mathcal{F}(x,y,z): x,y,z\in\Z_p\}=\Z_p$ for each prime $p$ (we will discuss the local representations in Section 2). We now state our main results.
\begin{theorem}\label{Thm. A}
Suppose that $\mathcal{F}$ has no local obstruction, and assume that $0<r<s$ with $r\equiv s\pmod2$. Then $\mathcal{F}$ is not almost universal if and only if all of the following are satisfied:
\medskip

{\rm (i)} $\begin{cases}r\ge2,a\equiv b\equiv c\pmod4&\mbox{if}\ 2\mid\nu_3(abc),\\r\equiv s\equiv0\pmod2&\mbox{if}\ 2\nmid\nu_3(abc).\end{cases}$
\medskip

{\rm (ii)} For each prime factor $p\ne3$ of $\mathfrak{sf}(abc)$, we have $p\equiv1\pmod4$ if $2\mid\nu_3(abc)$ and $p\equiv1\pmod3$ if $2\nmid\nu_3(abc)$.
\medskip

{\rm (iii)} If $2\mid\nu_3(abc)$, then $M_3\cong\langle u_1,3^{2i}u_2,3^{2j}u_3\rangle$ with $u_1,u_2,u_3$ $3$-adic units and $0\le i\le j$. If $2\nmid\nu_3(abc)$, then $M_3\cong\langle u_1,3^iu_2,3^ju_3\rangle$ with $u_1,u_2,u_3$ $3$-adic units satisfying $u_1\equiv u_2\equiv u_3\pmod{3\Z_3}$ and $0<i<j$.
\medskip

{\rm (iv)} $\mathfrak{sf}(abc)\equiv a+2^rb+2^sc\pmod{24}$, and the equation
$$24\mathcal{F}(x,y,z)+a+2^rb+2^sc=\mathfrak{sf}(abc)$$
has no integral solutions.
\end{theorem}

\begin{theorem}\label{Thm. A1}
Suppose that $\mathcal{F}$ has no local obstruction, and assume that $0<r<s$ with $r\not\equiv s\pmod2$ and $2\mid\nu_3(abc)$. Then $\mathcal{F}$ is not almost universal if and only if all of the following are satisfied:
\medskip

{\rm (i)} One of the conditions {\rm $(\alpha)$} and {\rm $(\beta)$} holds, where

{\rm $(\alpha)$} $r=1,s\ge4$, $ab\equiv1\pmod8$, $c(a+2b)\equiv\begin{cases}1\pmod8&\mbox{if}\ s=4,\\1,3\pmod8&\mbox{if}\ s\ge6.\end{cases}$

{\rm $(\beta)$} $r>2, ab\equiv1,3\pmod8$, $bc\equiv\begin{cases}1\pmod8&\mbox{if}\ s-r=1,3,\\1,3\pmod8&\mbox{if}\ s-r\ge5.\end{cases}$
\medskip

{\rm (ii)} For each prime factor $p\ne3$ of $\mathfrak{sf}(abc)$, the integer $-2$ is a quadratic residue modulo $p$.
\medskip

{\rm (iii)} $M_3\cong\langle u_1,3^{2i}u_2,3^{2j}u_3\rangle$ with $u_1,u_2,u_3$ $3$-adic units and $0\le i\le j$.
\medskip

{\rm (iv)} $\mathfrak{sf}(abc)\equiv a+2^rb+2^sc\pmod{24}$, and the equation
$$24\mathcal{F}(x,y,z)+a+2^rb+2^sc=\mathfrak{sf}(abc)$$
has no integral solutions.
\end{theorem}

\begin{theorem}\label{Thm. A2}
Suppose that $\mathcal{F}$ has no local obstruction, and assume that $0<r<s$ with $r\not\equiv s\pmod2$ and $2\nmid\nu_3(abc)$. Then $\mathcal{F}$ is not almost universal if and only if all of the following are satisfied:
\medskip

{\rm (i)} One of the conditions {\rm $(\alpha)$} and {\rm $(\beta)$} holds, where

{\rm $(\alpha)$} $r=1,s\ge4$, $ab\equiv3\pmod8$, $c(a+2b)\equiv\begin{cases}1\pmod8&\mbox{if}\ s=4,\\\pm1\pmod8&\mbox{if}\ s\ge6.\end{cases}$

{\rm $(\beta)$} $r>2$, $ab\equiv\pm((-1)^r+2)\pmod8$, $bc\equiv\begin{cases}3\pmod8&\mbox{if}\ s-r\le3,\\\pm3\pmod8&\mbox{if}\ s-r\ge5.\end{cases}$
\medskip

{\rm (ii)} For any prime factor $p\ne3$ of $\mathfrak{sf}(abc)$, the integer $-6$ is a quadratic residue modulo $p$.
\medskip

{\rm (iii)} $M_3\cong\langle u_1,3^iu_2,3^ju_3\rangle$ with $u_1,u_2,u_3$ $3$-adic units satisfying $u_1u_2\equiv (-1)^i\pmod{3\Z_3}$, $u_1u_3\equiv(-1)^j\pmod{3\Z_3}$ and $0<i<j$.
\medskip

{\rm (iv)} $\mathfrak{sf}(abc)\equiv a+2^rb+2^sc\pmod{24}$, and the equation
$$24\mathcal{F}(x,y,z)+a+2^rb+2^sc=\mathfrak{sf}(abc)$$
has no integral solutions.
\end{theorem}

\begin{theorem}\label{Thm. B}
Suppose that $\mathcal{F}$ has no local obstruction, and assume that $r=s>0$ and $2\nmid\nu_3(abc)$. Then $\mathcal{F}$ is not almost universal if and only if all of the following are satisfied:

{\rm (i)} $r$ is even and $b\not\equiv c\pmod4$.

{\rm (ii)} Every prime factor $\ne3$ of $\mathfrak{sf}(abc)$ is congruent to $1$ modulo $3$.

{\rm (iii)} $M_3\cong\langle u_1, 3^iu_2,3^ju_3\rangle$, where $0<i<j$ are positive integers and
$u_1,u_2,u_3$ are $3$-adic units with $u_1\equiv u_2\equiv u_3\pmod{3\Z_3}$.

{\rm (iv)} $\mathfrak{sf}(abc)\equiv a+2^rb+2^sc\pmod{24}$, and the equation
$$24\mathcal{F}(x,y,z)+a+2^rb+2^sc=\mathfrak{sf}(abc)$$
has no integral solutions.
\end{theorem}

\begin{theorem}\label{Thm. C}
Suppose that $\mathcal{F}$ has no local obstruction, and assume that $r=s>0$ and $2\mid\nu_3(abc)$.
Then $\mathcal{F}$ is not almost universal if and only if all of the following are satisfied:

{\rm (i)} $r\ge2$, $a\equiv b\pmod4$ and $b\equiv c\pmod8$.

{\rm (ii)} Every prime factor of $\mathfrak{sf}(abc)$ is congruent to $1$ modulo $4$.

{\rm (iii)} $M_3\cong\langle u_1,3^{2i}u_2,3^{2j}u_3\rangle$, where $0\le i\le j$ are integers and
$u_1,u_2,u_3$ are $3$-adic units.

{\rm (iv)} $\mathfrak{sf}(abc)\equiv a+2^rb+2^sc\pmod{24}$, and the equation
$$24\mathcal{F}(x,y,z)+a+2^rb+2^sc=\mathfrak{sf}(abc)$$
has no integral solutions.
\end{theorem}

Now we state our last result.
\begin{theorem}\label{Thm. D}
Suppose that $\mathcal{F}$ has no local obstruction, and assume $r=s=0$. Then $\mathcal{F}$ is almost universal.
\end{theorem}
In Section 2 we shall introduce some notations and make some preparations for our proofs. The detailed proofs will be given in Section 3.
\maketitle

\section{Notations and Some Preparations}
\setcounter{lemma}{0}
\setcounter{theorem}{0}
\setcounter{corollary}{0}
\setcounter{remark}{0}
\setcounter{equation}{0}
\setcounter{conjecture}{0}
In this paper, we adopt the language of quadratic spaces and lattices. Any unexplained notations can be found in \cite{C,Ki,Oto}. Let $p$ be an arbitrary prime. We first introduce the following notations.
\begin{itemize}
\item $\Q_p^{\times}=\{x\in\Q_p: x\ne 0\}$ and $\Q_p^{\times2}=\{x^2: x\in\Q_p^{\times}\}$.
\item $\Z_p^{\times}=\{x\in\Z_p: \text{$x$\ is\ invertible\ in $\Z_p$}\}$
and $\Z_p^{\times2}=\{x^2: x\in\Z_p^{\times}\}$.
\end{itemize}
Let $(V,\b,\q)$ be an arbitrary positive definite quadratic space over $\Q$ with symmetric bilinear map $\b$ and associated quadratic map $\q$. Given an arbitrary $\Z$-lattice $N$ contained in $V$,
we write $N\cong A$ if $A$ is the gram matrix for $N$ with respect to some basis. Furthermore,
a diagonal matrix with $a_1,a_2,\cdots,a_m$ as the diagonal entries is abbreviated as $\langle a_1,a_2,\cdots,a_m\rangle$. We also define
\begin{itemize}
\item $\q(N)=\{\q({\bf w}): {\bf w}\in N\}$ and $\q(N_p)=\{\q({\bf u}): {\bf u}\in N_p\}$.
\end{itemize}
A vector ${\bf u}\in N_p$ is called primitive if $p^{-1}\cdot{\bf u}\not\in N_p$. A vector ${\bf w}\in N$ is said to be primitive if it is primitive in $N_p$ for all primes $p$. We now define the following notations.
\begin{itemize}
\item $\q^*(N_p)=\{\q({\bf u}): \text{${\bf u}$\ is\ a\ primitive\ vector\ in\ $N_p$}\}$.
\item $\q^*(N)=\{\q({\bf w}): \text{${\bf w}$\ is\ a\ primitive\ vector\ in\ $N$}\}$.
\end{itemize}
For an arbitrary positive integer $t$, we say that $t$ can be represented by $\gen(N)$ if there is a lattice $N'\in\gen(N)$ such that $t\in \q(N')$. When this occurs, we write $t\in \q(\gen(N))$. Similarly, $t\in \q^*(\gen(N))$ means that there is a lattice $N''\in\gen(N)$ such that $t\in \q^*(N'')$. Also, when this occurs, we say that $t$ can be primitively represented by $\gen(M)$.
By \cite[Theorem 1.3, p. 129 and Theorem 5.1, p. 143]{C} we know that
$$t\in \q(\gen(N))\Leftrightarrow t\in \q(N_p)\ \text{for\ all\ primes\ $p$}.$$
and
$$t\in \q^*(\gen(N))\Leftrightarrow t\in \q^*(N_p)\ \text{for\ all\ primes\ $p$}.$$
Meanwhile, we say that $t\in\q(\spn(N))$ if there is a lattice $N_*\in\spn(N)$ such that $t\in\q(N_*)$. When this occurs, we say that $t$ can be represented by $\spn(N)$. Similarly, we say that $t$ can be primitively represented by $\spn(N)$ if there is a lattice $N_{**}\in\spn(N)$ such that $t\in\q^*(N_{**})$. When this happens, we write $t\in\q^*(\spn(N))$. For convenience, for each positive integer $n$ we let
\begin{itemize}
\item $l(n)=24n+a+2^rb+2^sc.$
\end{itemize}
As mentioned in Section 1, we see that $n$ can be represented by $\mathcal{F}$ if and only if there are $x,y,z\in\Z$ such that
$$l(n)=a(6x-1)^2+2^rb(6y-1)^2+2^sc(6z-1)^2.$$
With the notations in Section 1, recall that $(V,\b,\q)$ is a ternary quadratic space over $\Q$ with an orthogonal basis $\{{\bf e_1},{\bf e_2}, {\bf e_3}\}$ satisfying
$$\q({\bf e_1})=a,\ \q({\bf e_2})=2^rb,\ {\rm and}\ \q({\bf e_3})=2^sc.$$
Recall that
$$L=\Z{\bf e_1}+\Z{\bf e_2}+\Z{\bf e_3},$$
and
$${\bf v}=-({\bf e_1}+{\bf e_2}+{\bf e_3}).$$
As mentioned in Section 1 we obtain that $n$ can be represented by $\mathcal{F}$ if and only if
$$l(n)\in \q({\bf v}+6L),$$
where $\q({\bf v}+6L)=\{\q({\bf v}+6{\bf w}): {\bf w}\in L\}.$ Now we consider the $\Z$-lattice $M=\Z{\bf v}+6L$. Clearly $\{6{\bf e_1},6{\bf e_2}, {\bf v}\}$ is a $\Z$-basis of $M$ and the gram matrix of $M$ with respect to this basis is
$$\begin{pmatrix} 36a & 0 &-6a \\0 & 36\cdot2^rb &-6\cdot2^rb \\ -6a & -6\cdot2^rb &\ve \end{pmatrix},$$
where $\ve:=\q({\bf v})=a+2^rb+2^sc.$ The discriminant $dM$ of $M$ is equal to $6^4\cdot2^{r+s}abc$.

We now claim that
$$l(n)\in \q({\bf v}+6L)\Leftrightarrow l(n)\in \q(M).$$
It is enough to show the $`` \Leftarrow "$ part. In fact, suppose that $l(n)=\q(\lambda{\bf v}+6{\bf w})$ for some $\lambda\in\Z$ and some vector ${\bf w}\in L$. Then we have
$$\ve\equiv \lambda^2\ve\pmod 6.$$
Since $(6,\ve)=1$ by condition $(*)$, we have $\lambda\equiv\pm1\pmod6$. Noting that
$\q(\lambda{\bf v}+6{\bf w})=\q(-\lambda{\bf v}-6{\bf w})$, we therefore obtain $l(n)\in \q({\bf v}+6L)$.
This confirms our claim.

We now consider the local representation of $n$ by $\mathcal{F}$. Recall that $\mathcal{F}$ has no local obstruction if every positive integer can be represented by $\mathcal{F}$ over $\Z_p$ for all primes $p$. Clearly $\mathcal{F}$ has no local obstruction if and only if $l(n)\in \q(M_p)$ for all primes $p$ and all positive integers $n$. We begin with the following Lemma.
\begin{lemma}\label{Lemma local obstruction}
{\rm (i)} $\mathcal{F}(x,y,z)$ has no local obstruction if and only if for all primes $p\ne2,3$ we have
$$M_p\cong\langle1,-1,-dM\rangle.$$

{\rm (ii)} For all positive integers $n$, we have $l(n)\in \q^*(\gen(M))$.
\end{lemma}
\begin{proof}
{\rm (i)} We first consider the `` only if " part. Suppose that $\mathcal{F}$ has no local obstruction. Then $l(n)=24n+a+2^rb+2^sc\in \q(M_p)$ for all primes $p$. In particular, for each prime $p\ne 2,3$, we have $\q(M_p)=\Z_p$. This clearly implies that $M_p$ is isotropic for each prime $p\ne 2,3$. If $p\ne 2,3$ and $p\nmid d(M)$, then $M_p$ is a unimodular lattice. By \cite[92:1]{Oto} we have
$$M_p\cong\langle1,-1,-dM\rangle.$$
Assume now that $p\ne 2,3$ and $p\mid dM$. Since $1\in \q(M_p)$, by \cite[82:15]{Oto} we see that the sublattice $\langle1\rangle$ splits $M_p$, i.e., $M_p\cong\langle1\rangle\perp W_p$ for some binary lattice $W_p$. As $\Z_p^{\times}\subseteq \q(M_p)$, there must exist a unit $\ve_p\in\Z_p^{\times}$ such that $\ve_p\in \q(W_p)$. By \cite[82:15]{Oto} again we have
$$M_p\cong\langle1,\ve_p,\ve_pdM\rangle.$$
Since $p\in \q(M_p)$, it is easy to see that there are $x_p,y_p,z_p\in\Z_p$ with $x_p,y_p\in\Z_p^{\times}$ such that
$$p=x_p^2+\ve_py_p^2+\ve_pdMz_p^2.$$
This implies $0\equiv x_p^2+\ve_py_p^2\pmod{p\Z_p}$. By the local square theorem (cf. \cite[63:1]{Oto}) we obtain $-\ve_p\in\Z_p^{\times2}$ and hence
$$M_p\cong\langle1,-1,-dM\rangle.$$
This proves the `` only if " part.

We now consider the `` if " part. Suppose that
$$M_p\cong\langle1,-1,-dM\rangle$$
for all primes $p\ne2,3$. Clearly $\langle1,-1\rangle$ is a unimodular lattice over $\Z_p$ for each prime $p\ne2$. By \cite[92:1]{Oto} we have
$$\langle1,-1\rangle\cong\left(\begin{matrix}0 & 1\\ 1 & 0\end{matrix}\right).$$
Hence we obtain
$$M_p\cong\left(\begin{matrix}0 & 1\\ 1 & 0\end{matrix}\right)\perp\langle-dM\rangle.$$
Clearly this implies that $\q^*(M_p)=\Z_p$ for all primes $p\ne2,3$.

We now turn to the cases $p=2$ or $3$. Suppose that $p$ is either $2$ or $3$. As $(6,a+2^rb+2^sc)=1$, by the local square theorem we have
$$l(n)=24n+a+2^rb+2^sc\in(a+2^rb+2^sc)\Z_p^{\times2}.$$
for all positive integers $n$. Let $\eta_{p,n}\in\Z_p^{\times}$ such that $l(n)=(a+2^rb+2^sc)\eta_{p,n}^2$.
Then it is clear that
$$l(n)=\q(\eta_{p,n}\cdot{\bf v})\in \q^*(M_p).$$

{\rm (ii)} In view of the proof of {\rm (i)}, we see that {\rm (ii)} holds.

This completes the proof.
\end{proof}
We now give a brief discussion on spinor exceptions of $\gen(M)$. Suppose that a positive integer $t\in\q(\gen(M))$. We call $t$ a (primitive) spinor exception of $\gen(M)$ if $t$ is (primitively) represented by exactly half of the spinor genera in $\gen(M)$. In the fundamental paper \cite{Kn}, Kneser first investigated this topic. Later Schulze-Pillot \cite{SP80} obtained necessary and sufficient conditions for $t$ to be a spinor exception. Earnest, Hisa and Hung \cite{EHH} offered a characterization of primitive spinor exceptions, i.e., a positive integer $t\in\q(\gen(M))$ is a primitive spinor exception of $\gen(M)$ if and only if for all primes $p$ we have
\begin{equation}\label{Eq. A}
\theta(O^+(M_p))\subseteq {\rm N}_{\Q_p(\sqrt{-tdM})/\Q_p}(\Q_p(\sqrt{-tdM})^{\times})=\theta^*(M_p,t),
\end{equation}
where $\theta$ denotes the spinor norm map on the proper orthogonal group $O^+(M_p)$ and the symbol
${\rm N}_{\Q_p(\sqrt{-tdM})/\Q_p}(\Q_p(\sqrt{-tdM})^{\times})$ is the norm group of extension $\Q_p(\sqrt{-tdM})/\Q_p$, and $\theta^*(M_p,t)$ is the group generated by primitively relative spinor norms (for the precise definitions of these symbols, readers may consult \cite{EHH}).

Suppose that a square-free positive integer $t$ is a primitive spinor exception of $\gen(M)$ with $t\in\mathcal{A}$, where
$$\mathcal{A}:=\{24n+a+2^rb+2^sc: n\ge0\}.$$
Assume first $t\not\in\q(\spn(M))$. Let $p\nmid6abc$ be a prime that
splits in $\Q(\sqrt{-tdM})$. Then by \cite{SP00} we have $tp^2\not\in\q(\spn(M))$. Since there are infinitely many primes that split in $\Q(\sqrt{-tdM})$ and
$tp^2\in\mathcal{A}$ for all primes $p$ with $(p,6)=1$, we see that there are infinitely many positive integers $n$ such that $l(n)\not\in\q(M)$. Hence $\mathcal{F}$ is not almost universal. Assume now that
$t\in\q(\spn(M))$ and $t\not\in\q(M)$. Let $p\nmid6abc$ be a prime that is inert in $\Q(\sqrt{-tdM})$. By \cite{SP00} we know that $tp^2\not\in\q^*(\spn(M))$. As $t$ is square-free, we must have $tp^2\not\in\q(M)$. Since there are infinitely many primes that are inert in $\Q(\sqrt{-tdM})$, it is easy to see that $\mathcal{F}$ is not almost universal. On the other hand, by Lemma \ref{Lemma local obstruction} we know that
$$\mathcal{A}\subseteq\q^*(\gen(M)).$$
Suppose that every primitive spinor exception in $\mathcal{A}$ can be represented by $M$. Then by \cite[Corollary of Theorem 3]{DR} we obtain that $24n+a+2^rb+2^sc\in\q(M)$ if $n$ is sufficiently large. Hence $\mathcal{F}$ is almost universal.

We also need the following Lemma.
\begin{lemma}\label{Lemma quadratic field}
Suppose that $t\in\Z^+$ is a primitive spinor exception of $\gen(M)$. Then we have
$$\Q(\sqrt{-tdM})\in\{\Q(\sqrt{-D}): D=1,2,3,6\}.$$
\end{lemma}
\begin{proof}
As $t$ is a primitive spinor exception of $\gen(M)$, by (\ref{Eq. A}) we have
$\theta(O^+(M_p))\subseteq {\rm N}_{\Q_p(\sqrt{-tdM})/\Q_p}(\Q_p(\sqrt{-tdM})^{\times})$. By Lemma \ref{Lemma local obstruction} we know that
$$M_p\cong\langle1,-1,-dM\rangle$$
for all primes $p\ne2,3$. Hence for all primes $p\ne2,3$ we have
$$\Z_p^{\times}\subseteq\theta(O^+(M_p))\subseteq {\rm N}_{\Q_p(\sqrt{-tdM})/\Q_p}(\Q_p(\sqrt{-tdM})^{\times}).$$
By local class field theory (cf. \cite[Proposition 1.7, p. 323]{N}), we see that the extension $\Q_p(\sqrt{-tdM})/\Q_p$ is unramified for each prime $p\ne2,3$. Hence we must have
$$\Q(\sqrt{-tdM})\in\{\Q(\sqrt{-D}): D=1,2,3,6\}.$$
This completes the proof.
\end{proof}

\section{Proofs}
\setcounter{lemma}{0}
\setcounter{theorem}{0}
\setcounter{corollary}{0}
\setcounter{remark}{0}
\setcounter{equation}{0}
\setcounter{conjecture}{0}
Before the proofs of Theorems \ref{Thm. A}--\ref{Thm. C}, we first consider the Jordan decomposition of $M_2$ in the case $r>0$. By computation, one can verify that for any $r\in\Z^+$ we have
\begin{equation}\label{Eq. M2}
M_2\cong\langle\ve,\ 2^{r+2}\ve b(a+2^sc),\ 2^{s+2}ac(a+2^sc)\rangle,
\end{equation}
where $\ve=a+2^rb+2^sc$.

{\bf Proof of Theorem \ref{Thm. A}.} We first prove the `` if " part. Suppose that ${\rm (i)}-{\rm (iv)}$ are satisfied. We shall show that $t=\mathfrak{sf}(abc)$ is a primitive spinor exception of $\gen(M)$. Since $r\equiv s\pmod2$ and $dM=6^4\cdot2^{r+s}abc$, by Lemma \ref{Lemma quadratic field} we have $\Q(\sqrt{-tdM})=\Q(\sqrt{-1})$ if $2\mid\nu_3(abc)$ and $\Q(\sqrt{-tdM})=\Q(\sqrt{-3})$ if $2\nmid\nu_3(abc)$.
By Lemma \ref{Lemma local obstruction} we have $t\in\q^*(\gen(M))$ since $t\equiv a+2^rb+2^sc\pmod {24}$.
We first consider the lattice $M_2$. As $r\equiv s\pmod2$ and $r>0$, we have $s\ge3$. Hence by the local square theorem we have
$$M_2\cong\langle\ve,2^{r+2}ab\ve,2^{s+2}c\rangle.$$
We assume first that $2\mid\nu_3(abc)$. Then $\Q(\sqrt{-tdM})=\Q(\sqrt{-1})$. As $r>1$, by \cite[Theorem 2.7]{EH} the spinor norm group of $\theta(O^+(M_2))$ can be obtained by the following way. Let
$$U\cong\langle1,2^{r+2}ab\rangle,\ \text{and}\ W\cong2^{r+2}ab\langle1,2^{s-r}abc\ve\rangle.$$
Then we have
$$\theta(O^+(M_2))=\theta(O^+(M_2^{1/\ve}))=\q(\mathcal{P}(U))\q(\mathcal{P}(W))\Q_2^{\times2},$$
where $\mathcal{P}(U)$ (resp. $\mathcal{P}(W)$) is the set of primitive anisotropic vectors ${\bf z}$ in $U$ (resp. $W$) whose associated symmetries $\tau_{\bf z}$ are contained in $O(U)$ (resp. $O(W)$). We first consider $\q(\mathcal{P}(U))$. Suppose that ${\bf x},{\bf y}\in U$ is a $\Z_2$-basis of $U$ with $\q({\bf x})=1$ and $\q({\bf y})=2^{r+2}ab$. Then for any $\sigma\in O^+(U)$, by \cite[43:3b]{Oto} there is an anisotropic primitive vector ${\bf z}$ such that $\sigma=\tau_{\bf x}\tau_{\bf z}$. It is easy to see that $\tau_{\bf z}\in O(U)$, i.e., ${\bf z}\in\mathcal{P}(U)$. Hence we have $\theta(O^+(U))\subseteq\q(\mathcal{P}(U))\Q_2^{\times2}$. Conversely, for each ${\bf z}\in\mathcal{P}(U)$, we have $\tau_{\bf x}\tau_{\bf z}\in O^+(U)$ and hence  $\q(\mathcal{P}(U))\Q_2^{\times2}\subseteq\theta(O^+(U))$. We therefore have
$$\q(\mathcal{P}(U))\Q_2^{\times2}=\theta(O^+(U)).$$
Similarly, we also have
$$\q(\mathcal{P}(W))\Q_2^{\times2}=2^{r+2}ab\cdot\theta(O^+(W)).$$
Moreover, the formulae for the spinor norm group of non-modular binary $\Z_2$-lattice can be found in \cite[1.9]{EH75}. From this we obtain that $\q(\mathcal{P}(U))\Q_2^{\times2}$ is equal to
$$\theta(O^+(\langle1,2^{r+2}ab\rangle))
=\begin{cases}\Q_2^{\times2}\cup5\Q_2^{\times2}\cup ab\Q_2^{\times2}\cup5ab\Q_2^{\times2}&\mbox{if}\ r=2,\\\Q_2^{\times2}\cup 2^rab\Q_2^{\times2}&\mbox{if}\ r\ge3,\end{cases}$$
and that $\q(\mathcal{P}(W))\Q_2^{\times2}=2^rab\cdot\theta(O^+(\langle1,2^{s-r}abc\ve\rangle))$ is equal to
$$2^rab\cdot\begin{cases}\{\gamma\in\Z_2^{\times}\Q_2^{\times2}: (\gamma,-abc\ve)_2=1\}&\mbox{if}\ s-r=2,\\\Q_2^{\times2}\cup5\Q_2^{\times2}\cup abc\ve\Q_2^{\times2}\cup 5abc\ve\Q_2^{\times2}&\mbox{if}\ s-r=4,\\\Q_2^{\times2}\cup abc\ve\Q_2^{\times2}&\mbox{if}\ s-r\ge5,\end{cases}$$
where $(\ ,\ )_2$ is the quadratic Hilbert symbol over $\Q_2$ (readers may refer to \cite[Chapter V, Section 3]{N} for details). If {\rm (i)} holds, then one can easily verify that $\theta(O^+(M_2))\subseteq {\rm N}_{\Q_2(\sqrt{-1})/\Q_2}(\Q_2(\sqrt{-1})^{\times})$.
By \cite[Theorem 2(b)]{EHH} and the above, we further have
\begin{equation}\label{Eq. spinor M2 in Thm.A}
\theta(O^+(M_2))\subseteq {\rm N}_{\Q_2(\sqrt{-1})/\Q_2}(\Q_2(\sqrt{-1})^{\times})=\theta^*(M_2,t).
\end{equation}
Assume now that $2\nmid\nu_3(abc)$. Then $\Q(\sqrt{-tdM})=\Q(\sqrt{-3})$. As the extension $\Q_2(\sqrt{-3})/\Q_2$ is unramified, by \cite[Theorem 2(a)]{EHH} we have
\begin{equation}\label{Eq. spinor norm M_2 unramified case in Thm. A}
\theta(O^+(M_2))\subseteq {\rm N}_{\Q_2(\sqrt{-3})/\Q_2}(\Q_2(\sqrt{-3})^{\times})=\theta^*(M_2,t)
\end{equation}
if and only if $r\equiv s\equiv 0\pmod2$.

Now we consider the lattices $M_p$ for all primes $p\ne2,3$. If $p\nmid\mathfrak{sf}(abc)$ and $p\ne2,3$, then by Lemma \ref{Lemma local obstruction} we must have
$$M_p\cong\langle1,-1,p^{2j}u_p\rangle$$
for some $j\ge0$ and some $u_p\in\Z_p^{\times}$. Hence by \cite[Theorem 1(a)]{EHH} we have
\begin{equation}\label{Eq. spinor Mp in Thm.A}
\theta(O^+(M_p))\subseteq {\rm N}_{\Q_p(\sqrt{-tdM})/\Q_p}(\Q_p(\sqrt{-tdM})^{\times})=\theta^*(M_p,t).
\end{equation}
Assume now that $p>3$ and $p\mid\mathfrak{sf}(abc)$. Since ${\rm (ii)}$ holds, we have
$-tdM\in\Q_p^{\times2}$. When this occurs, we clearly have
\begin{equation}\label{Eq. spinor Mp 2 in Thm.A}
\theta(O^+(M_p))\subseteq{\rm N}_{\Q_p(\sqrt{-tdM})/\Q_p}(\Q_p(\sqrt{-tdM})^{\times})=\theta^*(M_p,t)=\Q_p^{\times}
\end{equation}
Now we consider the lattice $M_3$. We first assume $2\mid\nu_3(abc)$. Clearly $\Q(\sqrt{-tdM})=\Q(\sqrt{-1})$ and $\Q_3(\sqrt{-1})/\Q_3$ is unramified. Hence by \cite[Theorem 1(a)]{EHH} we obtain that
$$\theta(O^+(M_3))\subseteq{\rm N}_{\Q_3(\sqrt{-1})/\Q_3}(\Q_3(\sqrt{-1})^{\times})=\theta^*(M_3,t)$$ if and only if
$M_3\cong\langle u_1,3^{2i}u_2,3^{2j}u_3\rangle$ for some $u_1,u_2,u_3\in\Z_3^{\times}$ and $0\le i\le j$. Assume now $2\nmid\nu_3(abc)$. Then we have $\Q(\sqrt{-tdM})=\Q(\sqrt{-3})$. As
$$M_3\cong\langle u_1,3^iu_2,3^ju_3\rangle$$
with $0<i<j$ and $u_1,u_2,u_3\in\Z^{\times}_3$ satisfying $u_1\equiv u_2\equiv u_3\pmod{3\Z_3}$, by
\cite[Satz 3]{Kn56} we have
\begin{align*}
\theta(O^+(M_3))
&=\Q_3^{\times2}\cup3^iu_1u_2\Q_3^{\times2}\cup3^ju_1u_3\Q_3^{\times2}\cup3^{i+j}u_2u_3\Q_3^{\times2}\\
&\subseteq\{1,3\}\Q_3^{\times2}={\rm N}_{\Q_3(\sqrt{-3})/\Q_3}(\Q_3(\sqrt{-3})^{\times}).
\end{align*}
By \cite[Theorem 1(b)]{EHH} we further have
\begin{equation}\label{Eq. spinor M3 in Thm.A}
\theta(O^+(M_3))\subseteq{\rm N}_{\Q_3(\sqrt{-3})/\Q_3}(\Q_3(\sqrt{-3})^{\times})=\theta^*(M_3,t).
\end{equation}
In view of (\ref{Eq. spinor M2 in Thm.A})--(\ref{Eq. spinor M3 in Thm.A}), we obtain that $t=\mathfrak{sf}(abc)$ is a primitive spinor exception of $\gen(M)$ by (\ref{Eq. A}). Moreover, since the equation
$$24\mathcal{F}(x,y,z)+a+2^rb+2^sc=\mathfrak{sf}(abc)$$
has no integral solutions, we have $t\not\in\q(M)$. Hence by the discussion below the proof of Lemma \ref{Lemma local obstruction}, we see that $\mathcal{F}$ is not almost universal.

We now prove the `` only if " part. Suppose that $\mathcal{F}$ is not almost universal. Then by the discussion below the proof of Lemma \ref{Lemma local obstruction}, there must exist at least one primitive spinor exception in $\mathcal{A}=\{24n+a+2^rb+2^sc: n\ge0\}$. Let $t'\in\mathcal{A}$ be an arbitrary  primitive spinor exception of $\gen(M)$. By Lemma \ref{Lemma quadratic field} we have
$$\Q(\sqrt{-t'dM})=\begin{cases}\Q(\sqrt{-1})&\mbox{if}\ 2\mid\nu_3(abc),\\\Q(\sqrt{-3})&\mbox{if}\  2\nmid\nu_3(abc).\end{cases}$$
We now consider $M_2$. As in the `` if " part, we have
$$M_2\cong\langle\ve,2^{r+2}ab\ve,2^{s+2}c\rangle.$$
Assume first $2\nmid\nu_3(abc)$, i.e., $\Q(\sqrt{-t'dM})=\Q(\sqrt{-3})$. As $\Q_2(\sqrt{-3})/\Q_2$ is unramified, by \cite[Theorem 2(a)]{EHH} we have
$$\theta(O^+(M_2))\subseteq{\rm N}_{\Q_2(\sqrt{-3})/\Q_2}(\Q_2(\sqrt{-3})^{\times})=\theta^*(M_2,t)$$
if and only if the Jordan components of $M_2$ have either all even or all odd orders
(a $\Z_2$-lattice $N$ has even order if $\nu_2(\q({\bf x}))$ is even for every
primitive vector ${\bf x}\in N$ with its associated symmetry $\tau_{{\bf x}}\in O(N)$, and $N$ is
of odd order if $\nu_2(\q({\bf x}))$ are all odd for such ${\bf x}$). This is equivalent to $r\equiv s\equiv0\pmod2$. Hence the second part of {\rm (i)} holds.
Assume now that $2\mid\nu_3(abc)$. Then $\Q(\sqrt{-t'dM})=\Q(\sqrt{-1})$.
As $$M_2\cong\langle\ve,2^{r+2}ab\ve,2^{s+2}c\rangle.$$
If $r=1$, then by \cite[Theorem 2(b)(iv)]{EHH} we have either $\theta(O^+(M_2))\nsubseteq{\rm N}_{\Q_2(\sqrt{-1})/\Q_2}(\Q_2(\sqrt{-1})^{\times})$ or ${\rm N}_{\Q_2(\sqrt{-1})/\Q_2}(\Q_2(\sqrt{-1})^{\times})\ne\theta^*(M_2,t')$. This is a contradiction. Hence $r\ge2$. Adopting the notations in the `` if " part, we have
$$\theta(O^+(M_2))=\theta(O^+(M_2^{1/\ve}))=\q(\mathcal{P}(U))\q(\mathcal{P}(W))\Q_2^{\times2}.$$
Similar to the `` if " part,
$\q(\mathcal{P}(U))\Q_2^{\times2}$ is equal to
$$\theta(O^+(\langle1,2^{r+2}ab\rangle))
=\begin{cases}\Q_2^{\times2}\cup5\Q_2^{\times2}\cup ab\Q_2^{\times2}\cup5ab\Q_2^{\times2}&\mbox{if}\ r=2,\\\Q_2^{\times2}\cup 2^rab\Q_2^{\times2}&\mbox{if}\ r\ge3,\end{cases}$$
and that $\q(\mathcal{P}(W))\Q_2^{\times2}=2^rab\cdot\theta(O^+(\langle1,2^{s-r}abc\ve\rangle))$ is equal to
$$2^rab\cdot\begin{cases}\{\gamma\in\Z_2^{\times}\Q_2^{\times2}: (\gamma,-abc\ve)_2=1\}&\mbox{if}\ s-r=2,\\\Q_2^{\times2}\cup5\Q_2^{\times2}\cup abc\ve\Q_2^{\times2}\cup 5abc\ve\Q_2^{\times2}&\mbox{if}\ s-r=4,\\\Q_2^{\times2}\cup abc\ve\Q_2^{\times2}&\mbox{if}\ s-r\ge5.\end{cases}$$
As $\theta(O^+(M_2))\subseteq{\rm N}_{\Q_2(\sqrt{-1})/\Q_2}(\Q_2(\sqrt{-1})^{\times})$, we must have $ab\equiv1\pmod4$. By this it is easy to see that $\theta(O^+(M_2))\subseteq{\rm N}_{\Q_2(\sqrt{-1})/\Q_2}(\Q_2(\sqrt{-1})^{\times})$ if and only if $abc\ve\equiv1\pmod4$. As $r\ge2$, we obtain $ac\equiv1\pmod4$. This implies the first part of {\rm (i)}.

Now we consider $M_p$ with prime $p\ne3$ and $p\mid\mathfrak{sf}(abc)$. By Lemma \ref{Lemma local obstruction}, we have
$$M_p\cong\langle1,\ -1,\ p^{2j+1}u_p\rangle$$
for some $j\ge0$ and some $u_p\in\Z_p^{\times}$. If $-t'dM\not\in\Q_p^{\times2}$, then by \cite[Theorem 1(a)]{EHH} we have $\theta(O^+(M_p))\nsubseteq{\rm N}_{\Q_p(\sqrt{-t'dM})/\Q_p}(\Q_p(\sqrt{-t'dM})^{\times})$, which contradicts (\ref{Eq. A}). Hence $-t'dM\in\Q_p^{\times2}$ for all prime factor $p\ne3$ of $\mathfrak{sf}(abc)$. This implies ${\rm (ii)}$.

We now turn to $M_3$. Suppose first $2\mid\nu_3(abc)$. Then $\Q(\sqrt{-t'dM})=\Q(\sqrt{-1})$. As $\Q_3(\sqrt{-1})/\Q_3$ is unramified, by \cite[Theorem 1(a)]{EHH}
$$\theta(O^+(M_3))\subseteq{\rm N}_{\Q_3(\sqrt{-1})/\Q_3}(\Q_3(\sqrt{-1})^{\times})=\theta^*(M_3,t')$$
if and only if the first part of {\rm (iii)} holds. Suppose now $2\nmid\nu_3(abc)$. Then $\Q(\sqrt{-t'dM})=\Q(\sqrt{-3})$. By \cite[Theorem 1(b)]{EHH} we have
$$M_3\cong\langle u_1,3^iu_2,3^ju_3\rangle$$
with $u_1,u_2,u_3\in\Z_3^{\times}$ and $0<i<j$. By \cite[Satz 3]{Kn56} we have
$$\theta(O^+(M_3))
=\Q_3^{\times2}\cup3^iu_1u_2\Q_3^{\times2}\cup3^ju_1u_3\Q_3^{\times2}\cup3^{i+j}u_2u_3\Q_3^{\times2}.$$
As $\theta(O^+(M_3))\subseteq{\rm N}_{\Q_3(\sqrt{-3})/\Q_3}(\Q_3(\sqrt{-3})^{\times})=\{1,3\}\Q_3^{\times2}$, the second part of {\rm (iii)} holds.

Suppose now that $\mathfrak{sf}(abc)\not\equiv a+2^rb+2^sc\pmod{24}$, i.e., $\mathfrak{sf}(abc)\not\in\mathcal{A}$.
As mentioned above, we know that each primitive spinor exceptions $t'$ with $(6,t')=1$ are contained in $\mathfrak{sf}(abc)\Z^2$. Since $\mathfrak{sf}(abc)\not\in\mathcal{A}$, we have $t'\not\in\mathcal{A}$. By the discussion below the proof of Lemma \ref{Lemma local obstruction}, we see that $\mathcal{F}$ is almost universal, which is a contradiction. Hence we must have $\mathfrak{sf}(abc)\equiv a+2^rb+2^sc\pmod{24}$. Moreover, assume that the equation
$$24\mathcal{F}(x,y,z)+a+2^rb+2^sc=\mathfrak{sf}(abc)$$
has integral solutions.
Then it is clear that all primitive spinor exceptions in $\mathcal{A}$ can be represented by $M$. By the discussion below the proof of Lemma \ref{Lemma local obstruction}, we obtain that $\mathcal{F}$ is almost universal. This is a contradiction. In view of the above, ${\rm (iv)}$ holds.

This completes the proof of Theorem \ref{Thm. A}.\qed

{\bf Proof of Theorem \ref{Thm. A1}.} We first consider the `` if " part. Suppose that ${\rm (i)}-{\rm (iv)}$ are satisfied. We shall show that $t=\mathfrak{sf}(abc)$ is a primitive spinor exception of $\gen(M)$. Since $r\not\equiv s\pmod2$, $2\mid\nu_3(abc)$ and $dM=6^4\cdot2^{r+s}abc$, by Lemma \ref{Lemma quadratic field} we have $\Q(\sqrt{-tdM})=\Q(\sqrt{-2})$. By Lemma \ref{Lemma local obstruction} we have $t\in\q^*(\gen(M))$ since $t\equiv a+2^rb+2^sc\pmod {24}$.

We first consider $M_2$. By (\ref{Eq. M2}) we have
$$M_2^{1/\ve}\cong\langle1,2^{r+2}b(a+2^sc),2^{s+2}ac\ve(a+2^sc)\rangle.$$
Suppose first that {\rm ($\alpha$)} of {\rm (i)} holds. Then by the local square theorem we have
$$M_2^{1/\ve}\cong\langle1,2^{3}ab,2^{s+2}c\ve\rangle.$$
Let $U\cong\langle1,2^3ab\rangle$ and $W=2^3ab\langle1,2^{s-1}abc\ve\rangle$.
As in the proof of Theorem \ref{Thm. A}, we have
$$\theta(O^+(M_2))=\q(\mathcal{P}(U))\q(\mathcal{P}(W))\Q_2^{\times2}.$$
Also, one can verify that
$$\q(\mathcal{P}(U))\Q_2^{\times2}=\theta(O^+(U))=\{\gamma\in\Q_2^{\times}: (\gamma,-2ab)_2=1\},$$
and that $\q(\mathcal{P}(W))\Q_2^{\times2}$ is equal to
$$2ab\cdot\theta(O^+(W))=
2ab\cdot\begin{cases}\{\gamma\in\Q_2^{\times}: (\gamma,-2abc\ve)_2=1\}&\mbox{if}\ s=4,\\\Q_2^{\times2}\cup 2abc\ve\Q_2^{\times2}&\mbox{if}\ s\ge6.\end{cases}$$
Clearly $$\theta(O^+(M_2))=
\q(\mathcal{P}(U))\q(\mathcal{P}(W))\Q_2^{\times2}={\rm N}_{\Q_2(\sqrt{-2})/\Q_2}(\Q_2(\sqrt{-2})^{\times})$$
if {\rm ($\alpha$)} of {\rm (i)} holds. Assume now that {\rm ($\beta$)} of {\rm (i)} holds. Then
$$M_2^{1/\ve}\cong\langle1,2^{r+2}ab,2^{s+2}c\ve\rangle.$$
Letting $U\cong\langle1,2^{r+2}ab\rangle$, $W\cong2^{r+2}ab\langle1,2^{s-r}abc\ve\rangle$, one can easily verify that
$$\theta(O^+(M_2))=\q(\mathcal{P}(U))\q(\mathcal{P}(W))\Q_2^{\times2}.$$
Also, we have
$$\q(\mathcal{P}(U))\Q_2^{\times2}=\theta(O^+(U))=\Q_2^{\times2}\cup 2^rab\Q_2^{\times2},$$
and $\q(\mathcal{P}(W))\Q_2^{\times2}$ is equal to
$$2^rab\cdot\theta(O^+(W))=
2^rab\cdot\begin{cases}\{\gamma\in\Q_2^{\times}:(\gamma,-2abc\ve)_2=1\}&\mbox{if}\ s-r=1,3,\\\Q_2^{\times2}\cup 2abc\ve\Q_2^{\times2}&\mbox{if}\ s-r\ge5.\end{cases}$$
Clearly $$\theta(O^+(M_2))=\q(\mathcal{P}(U))\q(\mathcal{P}(W))\Q_2^{\times2}\subseteq{\rm N}_{\Q_2(\sqrt{-2})/\Q_2}(\Q_2(\sqrt{-2})^{\times})$$
if {\rm ($\beta$)} of {\rm (i)} holds. In view of the above, by \cite[Theorem 2(c)]{EHH} we have
\begin{equation}\label{Eq. spinor norm of M2 in Thm. A1}
\theta(O^+(M_2))\subseteq{\rm N}_{\Q_2(\sqrt{-2})/\Q_2}(\Q_2(\sqrt{-2})^{\times})=\theta^*(M_2,t).
\end{equation}

Now we consider $M_p$ with $p\ne2,3$ and $p\nmid\mathfrak{sf}(abc)$. With the essentially same reason as in the proof of Theorem \ref{Thm. A}, we have
\begin{equation}\label{Eq. spinor Mp in Thm.A1}
\theta(O^+(M_p))\subseteq {\rm N}_{\Q_p(\sqrt{-tdM})/\Q_p}(\Q_p(\sqrt{-tdM})^{\times})=\theta^*(M_p,t).
\end{equation}
Also, similar to the proof of Theorem \ref{Thm. A}, we see that for each prime factor $p\ne3$ of $\mathfrak{sf}(abc)$,
\begin{equation}\label{Eq. Mp 2 in Thm. A1}
\theta(O^+(M_p))\subseteq {\rm N}_{\Q_p(\sqrt{-tdM})/\Q_p}(\Q_p(\sqrt{-tdM})^{\times})=\theta^*(M_p,t),
\end{equation}
if and only if $-tdM\in\Q_p^{\times2}$, i.e., $-2$ is a quadratic residue modulo $p$.

Now we consider $M_3$. As $\Q_3(\sqrt{-2})/\Q_3$ is unramified, by \cite[Theorem 1(a)]{EHH} we have
\begin{equation}\label{Eq. M3 in Thm. A1}
\theta(O^+(M_3))\subseteq {\rm N}_{\Q_3(\sqrt{-tdM})/\Q_3}(\Q_3(\sqrt{-tdM})^{\times})=\theta^*(M_3,t),
\end{equation}
if and only if {\rm (iii)} holds. In view of (\ref{Eq. spinor norm of M2 in Thm. A1})--(\ref{Eq. M3 in Thm. A1}), we obtain that $t=\mathfrak{sf}(abc)$ is a primitive spinor exception of $\gen(M)$ by (\ref{Eq. A}). Moreover, since the equation
$$24\mathcal{F}(x,y,z)+a+2^rb+2^sc=\mathfrak{sf}(abc)$$
has no integral solutions, we have $t\not\in\q(M)$. Hence by the discussion below the proof of Lemma \ref{Lemma local obstruction}, we see that $\mathcal{F}$ is not almost universal.

We now show the `` only if " part. Suppose that $\mathcal{F}$ is not almost universal. Then by the discussion below the proof of Lemma \ref{Lemma local obstruction}, there must exist at least one primitive spinor exception in $\mathcal{A}=\{24n+a+2^rb+2^sc: n\ge0\}$. Let $t'\in\mathcal{A}$ be an arbitrary  primitive spinor exception of $\gen(M)$. By Lemma \ref{Lemma quadratic field} we have
$\Q(\sqrt{-t'dM})=\Q(\sqrt{-2})$. As $\Q_3(\sqrt{-2}/\Q_3)$ is unramified, by \cite[Theorem 1(a)]{EHH} we obtain that $\theta(O^+(M_3))\subseteq{\rm N}_{\Q_3(\sqrt{-2})/\Q_3}(\Q_3(\sqrt{-2})^{\times})$ implies that {\rm (iii)} holds. Moreover, for any prime factor $p\ne3$ of $\mathfrak{sf}(abc)$, by \cite[Theorem 1(a)]{EHH} again $\theta(O^+(M_p))\subseteq{\rm N}_{\Q_p(\sqrt{-2})/\Q_p}(\Q_p(\sqrt{-2})^{\times})$ implies $-t'dM\in\Q_p^{\times2}$, i.e., {\rm (ii)} holds.

Now we consider $M_2$. By (\ref{Eq. M2}) we have
$$M_2^{1/\ve}\cong\langle1,2^{r+2}b(a+2^sc),2^{s+2}ac\ve(a+2^sc)\rangle.$$
Let $$U\cong\langle1,2^{r+2}b(a+2^sc)\rangle,\ \text{and}\  W\cong2^{r+2}b(a+2^sc)\langle1,2^{s-r}abc\ve\rangle.$$
Suppose first $r>2$. Then by \cite[Theorem 2.7]{EH} we have
$\theta(O^+(M_2))=\q(\mathcal{P}(U))\q(\mathcal{P}(W))\Q_2^{\times2}\subseteq\{1,2,3,6\}\Q_2^{\times2}$. As in the `` if " part, we have $$\q(\mathcal{P}(U))\Q_2^{\times2}=\theta(O^+(U))=\Q_2^{\times2}\cup 2^rab\Q_2^{\times2},$$
and $\q(\mathcal{P}(W))\Q_2^{\times2}$ is equal to
$$2^rab\cdot\theta(O^+(W))=
2^rab\cdot\begin{cases}\{\gamma\in\Q_2^{\times}:(\gamma,-2abc\ve)_2=1\}&\mbox{if}\ s-r=1,3,\\\Q_2^{\times2}\cup 2abc\ve\Q_2^{\times2}&\mbox{if}\ s-r\ge5.\end{cases}$$
This gives $2^rab\in\{1,2,3,6\}\Q_2^{\times2}$. Hence $ab\equiv1,3\pmod8$. By this $\q(\mathcal{P}(U))\q(\mathcal{P}(W))\Q_2^{\times2}\subseteq\{1,2,3,6\}\Q_2^{\times2}$ implies that
$abc\ve\equiv1\pmod8$ if $s-r=1,3$, and that $abc\ve\equiv1,3\pmod8$ if $s-r\ge5$. This exactly shows that {\rm $(\beta)$} of {\rm (i)} holds.

Suppose now $r=2$. If $s=3$ or $5$, then by \cite[1.1]{EH} we have $\theta(O^+(M_2))=\Q_2^{\times}\not\subseteq\{1,2,3,6\}\Q_2^{\times2}$. Hence $s\ge7$ is odd. By \cite[Theorem 2.7]{EH} we also have
$\theta(O^+(M_2))=\q(\mathcal{P}(U))\q(\mathcal{P}(W))\Q_2^{\times2}\subseteq\{1,2,3,6\}\Q_2^{\times2}$. Note that
$$\q(\mathcal{P}(U))\Q_2^{\times2}=\theta(O^+(U))=\Q_2^{\times2}\cup ab\Q_2^{\times2}\cup 5ab\Q_2^{\times2}\cup 5\Q_2^{\times2},$$
and that $\q(\mathcal{P}(W))\Q_2^{\times2}$ is equal to
$$2^4ab\cdot\theta(O^+(W))=
ab\Q_2^{\times2}\cup 2c\ve\Q_2^{\times2}.$$
This gives $5\in\{1,2,3,6\}\Q_2^{\times2}$, which is a contradiction. Hence $r\ne2$.

Now we consider the case $r=1$. We have
$$M_2^{1/\ve}\cong\langle1,2^3b(a+2^sc),2^{s+2}ac\ve(a+2^sc)\rangle.$$
If $s=2$, then by \cite[1.1]{EH} we have $\theta(O^+(M_2))=\Q_2^{\times}\not\subseteq\{1,2,3,6\}\Q_2^{\times2}$. Hence $s\ge4$ is even. Similar to the above, we have
$$\q(\mathcal{P}(U))=\theta(O^+(U))=\{\gamma\in\Q_2^{\times}: (\gamma,-2ab)_2\}$$
and
$\q(\mathcal{P}(W))\Q_2^{\times2}$ is equal to
$$2^3ab\cdot\theta(O^+(W))=
2^3ab\cdot\begin{cases}\{\gamma\in\Q_2^{\times}:(\gamma,-2abc\ve)_2=1\}&\mbox{if}\ s=4,\\\Q_2^{\times2}\cup 2abc\ve\Q_2^{\times2}&\mbox{if}\ s\ge6.\end{cases}$$
This gives $2ab\in\{1,2,3,6\}\Q_2^{\times2}$. Hence $ab\equiv1,3\pmod8$. If $ab\equiv3\pmod8$, then $-1\in\q(\mathcal{P}(U))$ and hence $-ab\in\{1,2,3,6\}$. This is a contradiction. Hence we must have
$ab\equiv1\pmod8$. From this $\q(\mathcal{P}(U))\q(\mathcal{P}(W))\Q_2^{\times2}\subseteq\{1,2,3,6\}\Q_2^{\times2}$ implies that
$abc\ve\equiv1\pmod8$ if $s=4$, and that $abc\ve\equiv1,3\pmod8$ if $s\ge6$. This precisely shows that {\rm $(\alpha)$} of {\rm (i)} holds. With the same reason as in the proof of Theorem \ref{Thm. A}, the condition {\rm(iv)} holds.

In view of the above, we complete the proof. \qed

{\bf Proof of Theorem \ref{Thm. A2}.} We first prove the `` if " part. As in the proof of Theorem \ref{Thm. A}, we shall show that $t=\mathfrak{sf}(abc)$ is a primitive spinor exception of $\gen(M)$. Since $r\not\equiv s\pmod2$ and $2\nmid \nu_3(abc)$, by Lemma \ref{Lemma quadratic field} we have $\Q(\sqrt{-tdM})=\Q(\sqrt{-6})$.

For each $p\ne2,3$, with the essentially same method as in the proof of Theorem \ref{Thm. A1} we see that if {\rm (ii)} holds, then we have
\begin{equation}\label{Eq. Spinor Mp in Thm.A2}
\theta(O^+(M_p))\subseteq{\rm N}_{\Q_p(\sqrt{-6})/\Q_p}(\Q_p(\sqrt{-6})^{\times})=\theta^*(M_p,t).
\end{equation}
We turn to $M_3$. If {\rm (iii)} holds, then by \cite[Satz 13]{Kn56} we have
\begin{align*}
\theta(O^+(M_3))
&=\Q_3^{\times2}\cup3^iu_1u_2\Q_3^{\times2}\cup3^ju_1u_3\Q_3^{\times2}\cup3^{i+j}u_2u_3\Q_3^{\times2}\\
&\subseteq\{1,-3\}\Q_3^{\times2}={\rm N}_{\Q_3(\sqrt{-6})/\Q_3}(\Q_3(\sqrt{-6})^{\times}).
\end{align*}
As $\Q_3(\sqrt{-6})/\Q_3$ is ramified, by \cite[Theorem 1(b)]{EHH} we obtain
\begin{equation}\label{Eq. spinor M3 in Thm.A2}
\theta(O^+(M_3))\subseteq{\rm N}_{\Q_3(\sqrt{-6})/\Q_3}(\Q_3(\sqrt{-6})^{\times})=\theta^*(M_3,t).
\end{equation}
Now we consider $M_2$. By (\ref{Eq. M2}) we have
$$M_2^{1/\ve}\cong\langle1,2^{r+2}b(a+2^sc),2^{s+2}ac\ve(a+2^sc)\rangle.$$
Suppose first that {\rm ($\alpha$)} of {\rm (i)} holds. Then by the local square theorem we have
$$M_2^{1/\ve}\cong\langle1,2^{3}ab,2^{s+2}c\ve\rangle.$$
Let $U\cong\langle1,2^3ab\rangle$ and $W=2^3ab\langle1,2^{s-1}abc\ve\rangle$.
As in the proof of Theorem \ref{Thm. A}, we have
$$\theta(O^+(M_2))=\q(\mathcal{P}(U))\q(\mathcal{P}(W))\Q_2^{\times2}.$$
Also, one can verify that
$$\q(\mathcal{P}(U))\Q_2^{\times2}=\theta(O^+(U))=\{\gamma\in\Q_2^{\times}: (\gamma,-2ab)_2=1\},$$
and that $\q(\mathcal{P}(W))\Q_2^{\times2}$ is equal to
$$2ab\cdot\theta(O^+(W))=
2ab\cdot\begin{cases}\{\gamma\in\Q_2^{\times}: (\gamma,-2abc\ve)_2=1\}&\mbox{if}\ s=4,\\\Q_2^{\times2}\cup 2abc\ve\Q_2^{\times2}&\mbox{if}\ s\ge6.\end{cases}$$
Clearly $$\theta(O^+(M_2))=\q(\mathcal{P}(U))\q(\mathcal{P}(W))\Q_2^{\times2}\subseteq\{\pm1,\pm6\}\Q_2^{\times2}$$
if {\rm ($\alpha$)} of {\rm (i)} holds. Assume now that {\rm ($\beta$)} of {\rm (i)} holds. Then
$$M_2^{1/\ve}\cong\langle1,2^{r+2}ab,2^{s+2}c\ve\rangle.$$
Letting $U\cong\langle1,2^{r+2}ab\rangle$, $W\cong2^{r+2}ab\langle1,2^{s-r}abc\ve\rangle$, one can easily verify that
$$\theta(O^+(M_2))=\q(\mathcal{P}(U))\q(\mathcal{P}(W))\Q_2^{\times2}.$$
Also, we have
$$\q(\mathcal{P}(U))\Q_2^{\times2}=\theta(O^+(U))=\Q_2^{\times2}\cup 2^rab\Q_2^{\times2},$$
and $\q(\mathcal{P}(W))\Q_2^{\times2}$ is equal to
$$2^rab\cdot\theta(O^+(W))=
2^rab\cdot\begin{cases}\{\gamma\in\Q_2^{\times}:(\gamma,-2abc\ve)_2=1\}&\mbox{if}\ s-r=1,3,\\\Q_2^{\times2}\cup 2abc\ve\Q_2^{\times2}&\mbox{if}\ s-r\ge5.\end{cases}$$
Clearly $$\theta(O^+(M_2))=\q(\mathcal{P}(U))\q(\mathcal{P}(W))\Q_2^{\times2}\subseteq\{\pm1,\pm6\}\Q_2^{\times2}$$
if {\rm ($\beta$)} of {\rm (i)} holds. In view of the above, by \cite[Theorem 2(c)]{EHH} we have
\begin{equation}\label{Eq. spinor norm of M2 in Thm. A2}
\theta(O^+(M_2))\subseteq{\rm N}_{\Q_2(\sqrt{-6})/\Q_2}(\Q_2(\sqrt{-6})^{\times})=\theta^*(M_2,t).
\end{equation}
In view of (\ref{Eq. Spinor Mp in Thm.A2})--(\ref{Eq. spinor norm of M2 in Thm. A2}), we obtain that $t=\mathfrak{sf}(abc)$ is a primitive spinor exception of $\gen(M)$ by (\ref{Eq. A}). Moreover, since the equation
$$24\mathcal{F}(x,y,z)+a+2^rb+2^sc=\mathfrak{sf}(abc)$$
has no integral solutions, we have $t\not\in\q(M)$. Hence by the discussion below the proof of Lemma \ref{Lemma local obstruction}, we see that $\mathcal{F}$ is not almost universal.

Now we turn to the ``only if " part. Given any $p\ne2,3$ with $p\mid\mathfrak{sf}(abc)$, as in the proof of Theorem \ref{Thm. A1}, $\theta(O^+(M_p))\subseteq{\rm N}_{\Q_p(\sqrt{-6})/\Q_p}(\Q_p(\sqrt{-6})^{\times})=\theta^*(M_p,t)$ implies that {\rm (ii)} holds.
Moreover, as $\Q_3(\sqrt{-6})/\Q_3$ is ramified, by \cite[Theorem 1(b)]{EHH} there exist $u_1,u_2,u_3\in\Z_3^{\times}$ and $0<i<j$ such that
$M_3\cong\langle u_1,3^iu_2,3^ju_3\rangle.$ By \cite[Satz 13]{Kn56} we have
\begin{align*}
\theta(O^+(M_3))
=\Q_3^{\times2}\cup3^iu_1u_2\Q_3^{\times2}\cup3^ju_1u_3\Q_3^{\times2}\cup3^{i+j}u_2u_3\Q_3^{\times2}.
\end{align*}
By this $\theta(O^+(M_3))\subseteq\{1,-3\}\Q_3^{\times}$ implies that {\rm (iii)} holds.

Now we consider $M_2$. Note that ${\rm N}_{\Q_2(\sqrt{-6})/\Q_2}(\Q_2(\sqrt{-6})^{\times})=\{\pm1,\pm6\}\Q_2^{\times2}$.
By (\ref{Eq. M2}) we have
$$M_2^{1/\ve}\cong\langle1,2^{r+2}b(a+2^sc),2^{s+2}ac\ve(a+2^sc)\rangle.$$
Let $$U\cong\langle1,2^{r+2}b(a+2^sc)\rangle,\ \text{and}\  W\cong2^{r+2}b(a+2^sc)\langle1,2^{s-r}abc\ve\rangle.$$
Suppose first $r>2$. Then by \cite[Theorem 2.7]{EH} we have
$\theta(O^+(M_2))=\q(\mathcal{P}(U))\q(\mathcal{P}(W))\Q_2^{\times2}\subseteq\{\pm1,\pm6\}\Q_2^{\times2}$. As in the `` if " part, we have $$\q(\mathcal{P}(U))\Q_2^{\times2}=\theta(O^+(U))=\Q_2^{\times2}\cup 2^rab\Q_2^{\times2},$$
and $\q(\mathcal{P}(W))\Q_2^{\times2}$ is equal to
$$2^rab\cdot\theta(O^+(W))=
2^rab\cdot\begin{cases}\{\gamma\in\Q_2^{\times}:(\gamma,-2abc\ve)_2=1\}&\mbox{if}\ s-r=1,3,\\\Q_2^{\times2}\cup 2abc\ve\Q_2^{\times2}&\mbox{if}\ s-r\ge5.\end{cases}$$
This gives $2^rab\in\{\pm1,\pm6\}\Q_2^{\times2}$. Hence $ab\equiv\pm((-1)^r+2)\pmod8$. By this $\q(\mathcal{P}(U))\q(\mathcal{P}(W))\Q_2^{\times2}\subseteq\{\pm1,\pm6\}\Q_2^{\times2}$ implies that
$abc\ve\equiv3\pmod8$ if $s-r=1,3$, and that $abc\ve\equiv\pm3\pmod8$ if $s-r\ge5$. This exactly shows that {\rm $(\beta)$} of {\rm (i)} holds.

Suppose now $r=2$. If $s=3$ or $5$, then by \cite[1.1]{EH} we have $\theta(O^+(M_2))=\Q_2^{\times}\not\subseteq\{\pm1,\pm6\}\Q_2^{\times2}$. Hence $s\ge7$ is odd. By \cite[Theorem 2.7]{EH} we also have
$\theta(O^+(M_2))=\q(\mathcal{P}(U))\q(\mathcal{P}(W))\Q_2^{\times2}\subseteq\{\pm1,\pm6\}\Q_2^{\times2}$. Note that
$$\q(\mathcal{P}(U))\Q_2^{\times2}=\theta(O^+(U))=\Q_2^{\times2}\cup ab\Q_2^{\times2}\cup 5ab\Q_2^{\times2}\cup 5\Q_2^{\times2},$$
and that $\q(\mathcal{P}(W))\Q_2^{\times2}$ is equal to
$$2^4ab\cdot\theta(O^+(W))=
ab\Q_2^{\times2}\cup 2c\ve\Q_2^{\times2}.$$
This gives $5\in\{\pm1,\pm6\}\Q_2^{\times2}$, which is a contradiction. Hence $r\ne2$.

Now we consider the case $r=1$. We have
$$M_2^{1/\ve}\cong\langle1,2^3b(a+2^sc),2^{s+2}ac\ve(a+2^sc)\rangle.$$
If $s=2$, then by \cite[1.1]{EH} we have $\theta(O^+(M_2))=\Q_2^{\times}\not\subseteq\{\pm1,\pm6\}\Q_2^{\times2}$. Hence $s\ge4$ is even. Similar to the above, we have
$$\q(\mathcal{P}(U))=\theta(O^+(U))=\{\gamma\in\Q_2^{\times}: (\gamma,-2ab)_2\}$$
and
$\q(\mathcal{P}(W))\Q_2^{\times2}$ is equal to
$$2^3ab\cdot\theta(O^+(W))=
2^3ab\cdot\begin{cases}\{\gamma\in\Q_2^{\times}:(\gamma,-2abc\ve)_2=1\}&\mbox{if}\ s=4,\\\Q_2^{\times2}\cup 2abc\ve\Q_2^{\times2}&\mbox{if}\ s\ge6.\end{cases}$$
This gives $2ab\in\{\pm1,\pm6\}\Q_2^{\times2}$ and hence $ab\equiv\pm3\pmod8$. If $ab\equiv-3\pmod8$, then $3\in\q(\mathcal{P}(U))$ and hence $-2\in\{\pm1,\pm6\}\Q_2^{\times2}$. This is a contradiction. Hence $ab\equiv3\pmod8$. From this $\q(\mathcal{P}(U))\q(\mathcal{P}(W))\Q_2^{\times2}\subseteq\{\pm1,\pm6\}\Q_2^{\times2}$ implies that
$abc\ve\equiv3\pmod8$ if $s=4$, and that $abc\ve\equiv\pm3\pmod8$ if $s\ge6$. This precisely shows that {\rm $(\alpha)$} of {\rm (i)} holds. Finally, with the same reason as in the proof of Theorem \ref{Thm. A}, the condition {\rm(iv)} holds.

This completes the proof. \qed

{\bf Proof of Theorem \ref{Thm. B}.} We first prove the `` if " part. As in the proof of Theorem \ref{Thm. A}, we shall show that $t=\mathfrak{sf}(abc)$ is a primitive spinor exception of $\gen(M)$. Since $r=s$ and $2\nmid \nu_3(abc)$, by Lemma \ref{Lemma quadratic field} we have $\Q(\sqrt{-tdM})=\Q(\sqrt{-3})$.

We first focus on $M_2$. By (\ref{Eq. M2}) we have
$$M_2^{1/\ve}\cong\langle1\rangle\perp2^{r+2}b(a+2^rc)\langle1,abc\ve\rangle,$$
where $\ve=a+2^rb+2^sc$. As $r>0$ is even and $bc\equiv3\pmod4$, we see that the lattice $2^{r+2}b(a+2^sc)\langle1,abc\ve\rangle$ has even order.
Since $\Q_2(\sqrt{-3})/\Q_2$ is unramified, by \cite[Theorem 2(a)]{EHH} we have
$$\theta(O^+(M_2))\subseteq{\rm N}_{\Q_2(\sqrt{-3})/\Q_2}(\Q_2(\sqrt{-3})^{\times})=\theta^*(M_2,t)$$

We now consider the lattices $M_p$ with $p\ne2,3$. With the essentially same method in the proof of Theorem \ref{Thm. A}. We have
$$\theta(O^+(M_p))\subseteq{\rm N}_{\Q_p(\sqrt{-3})/\Q_p}(\Q_p(\sqrt{-3})^{\times})=\theta^*(M_p,t)$$
for all $p\nmid\mathfrak{sf}(abc)$. And for all primes $p\mid\mathfrak{sf}(abc)$
$$\theta(O^+(M_p))\subseteq{\rm N}_{\Q_p(\sqrt{-3})/\Q_p}(\Q_p(\sqrt{-3})^{\times})=\theta^*(M_p,t)$$
if and only if $-3\in\Q_p^{\times2}$, i.e., $p\equiv1\pmod3$.

We now turn to $M_3$. As $M_3\cong\langle u_1,3^iu_2,3^ju_3\rangle$ satisfies that $u_1,u_2,u_3\in\Z_3^{\times}$, $u_1\equiv u_2\equiv u_3\pmod{3\Z_3}$ and $0<i<j$,
by \cite[Satz 3]{Kn56} we have
$$\theta(O^+(M_3))=
\Q_3^{\times2}\cup 3^iu_1u_2\Q_3^{\times2}\cup3^ju_1u_3\Q_3^{\times2}\cup3^{i+j}u_2u_3\Q_3^{\times2}\subseteq\{1,3\}\Q_3^{\times2}.$$
According to \cite[Theorem 1(b)]{EHH}, we obtain
$$\theta(O^+(M_3))\subseteq{\rm N}_{\Q_3(\sqrt{-3})/\Q_3}=\theta^*(M_3,t).$$
In view of the above, we see that $t=\mathfrak{sf}(abc)$ is a primitive spinor exception of $\gen(M)$. Now with the same method in the proof of Theorem \ref{Thm. A}, we obtain that $\mathcal{F}$ is not almost universal.

We now consider the `` only if " part. Suppose that $\mathcal{F}$ is not almost universal. Clearly there is at least one primitive spinor exception in $\mathcal{A}$. Let $t'\in\mathcal{A}$ be any primitive spinor exception of $\gen(M)$. Similar to the proof of Theorem \ref{Thm. A}, it is easy to see that $\Q(\sqrt{-t'dM})=\Q(\sqrt{-3})$ and $t'\in\mathfrak{sf}(abc)\Z^2$. Similar to the proof of the `` if " part, we obtain that
$$\theta(O^+(M_2))\subseteq{\rm N}_{\Q_2(\sqrt{-3})/\Q_2}(\Q_2(\sqrt{-3})^{\times})=\theta^*(M_2,t')$$
if and only if the lattice $2^{r+2}b(a+2^rc)\langle1,abc\ve\rangle$ is even. This implies that $r$ is even and $b\not\equiv c\pmod4$, Also, for all primes $p\ne3$ with $p\mid\mathfrak{sf}(abc)$
$$\theta(O^+(M_p))\subseteq{\rm N}_{\Q_p(\sqrt{-3})/\Q_p}(\Q_p(\sqrt{-3})^{\times})=\theta^*(M_p,t')$$
if and only if $-3\in\Q_p^{\times2}$. Hence {\rm (i)} and {\rm (ii)} hold. For the lattice $M_3$,
by \cite[Theorem 1(b)]{EHH} we must have
$$M_3\cong\langle u_1,3^iu_2,3^ju_3\rangle,$$
for some $u_1,u_2,u_3\in\Z_3^{\times}$ and some $0<i<j$. Hence by \cite[Satz 3]{Kn56} we have
$$\theta(O^+(M_3))=
\Q_3^{\times2}\cup 3^iu_1u_2\Q_3^{\times2}\cup3^ju_1u_3\Q_3^{\times2}\cup3^{i+j}u_2u_3\Q_3^{\times2}.$$
As ${\rm N}_{\Q_3(\sqrt{-3})/\Q_3}(\Q_3(\sqrt{-3})^{\times})=\{1,3\}\Q_3^{\times2}$,
one can verify that $\theta(O^+(M_3))\subseteq{\rm N}_{\Q_3(\sqrt{-3})/\Q_3}(\Q_3(\sqrt{-3})^{\times})$ if and only if {\rm(iii)} holds. Finally, similar to the proof of Theorem \ref{Thm. A}, the condition {\rm(iv)} holds.

This completes the proof.\qed

{\bf Proof of Theorem \ref{Thm. C}.} We first prove the `` if " part. As in the proof of Theorem \ref{Thm. A}, we shall show that $t=\mathfrak{sf}(abc)$ is a primitive spinor exception of $\gen(M)$. Since $r=s$ and $2\mid \nu_3(abc)$, by Lemma \ref{Lemma quadratic field} we have $\Q(\sqrt{-tdM})=\Q(\sqrt{-1})$.

We first focus on $M_2$. By (\ref{Eq. M2}) we have
$$M_2^{1/\ve}\cong\langle1\rangle\perp2^{r+2}b(a+2^rc)\langle1,abc\ve\rangle.$$
As {\rm(i)} holds, we obtain that $b(a+2^rc)\langle1,abc\ve\rangle\cong\begin{pmatrix} \alpha & 1 \\1 & 2\beta\end{pmatrix}$ for some $\alpha,\beta\in\Z_2^{\times}$ and
$(2^{r+2}b(a+2^rc),-abc\ve)_2=1$.
Hence $M_2$ satisfies \cite[1.2(3)]{EH}. By \cite[1.2(3)]{EH} we obtain that
\begin{align*}
\theta(O^+(M_2))&=\{x\in\Q_2^{\times}: (x,-abc\ve)_2=1\}=\{x\in\Q_2^{\times}: (x,-1)_2=1\}\\
&={\rm N}_{\Q_2(\sqrt{-1})/\Q_2}(\Q_2(\sqrt{-1})^{\times}).
\end{align*}

We now consider the lattices $M_p$ with $p\ne2,3$. With the essentially same method in the proof of Theorem \ref{Thm. A}. We have
$$\theta(O^+(M_p))\subseteq{\rm N}_{\Q_p(\sqrt{-1})/\Q_p}(\Q_p(\sqrt{-1})^{\times})=\theta^*(M_p,t)$$
for all $p\nmid\mathfrak{sf}(abc)$. And for all primes $p\mid\mathfrak{sf}(abc)$
$$\theta(O^+(M_p))\subseteq{\rm N}_{\Q_p(\sqrt{-1})/\Q_p}(\Q_p(\sqrt{-1})^{\times})=\theta^*(M_p,t)$$
if and only if $-1\in\Q_p^{\times2}$, i.e., $p\equiv1\pmod4$. Moreover, As $\Q_3(\sqrt{-1})/\Q_3$ is unramified, by \cite[Theorem 1(a)]{EHH} we see that
$$\theta(O^+(M_3))\subseteq{\rm N}_{\Q_3(\sqrt{-1})/\Q_3}(\Q_3(\sqrt{-1})^{\times})=\theta^*(M_3,t)$$
if and only if {\rm (iii)} holds. With the same method in the proof of Theorem \ref{Thm. A}, we see that $\mathcal{F}$ is not almost universal.

We now consider the `` only if " part. Suppose that $\mathcal{F}$ is not almost universal. Then there is at least one primitive spinor exception in $\mathcal{A}$. Let $t'\in\mathcal{A}$ be any primitive spinor exception. In view of the above, it is easy to see that $\Q(\sqrt{-t'dM})=\Q(\sqrt{-1})$ and $t'\in\mathfrak{sf}(abc)\Z^2$. Similar to the proof of the `` if " part, we obtain that
for all primes $p\mid\mathfrak{sf}(abc)$
$$\theta(O^+(M_p))\subseteq{\rm N}_{\Q_p(\sqrt{-1})/\Q_p}(\Q_p(\sqrt{-1})^{\times})=\theta^*(M_p,t')$$
if and only if $-1\in\Q_p^{\times2}$, i.e., ${\rm(ii)}$ holds,
and that
$$\theta(O^+(M_3))\subseteq{\rm N}_{\Q_3(\sqrt{-1})/\Q_3}(\Q_3(\sqrt{-1})^{\times})=\theta^*(M_3,t')$$
if and only if {\rm (iii)} holds.

We now turn to
$$M_2^{1/\ve}\cong\langle1\rangle\perp2^{r+2}b(a+2^rc)\langle1,abc\ve\rangle.$$
The spinor norms of lattice of this type can be obtained by \cite[1.2]{EH}. According to \cite[1.2]{EH} it is easy to see that if $M_2$ does not satisfy \cite[1.2(3)]{EH}, then
$\theta(O^+(M_2))=\Q_2^{\times}$ or $\Z_2^{\times}\Q_2^{\times2}$, which is not contained in the norm group ${\rm N}_{\Q_2(\sqrt{-1})/\Q_2}(\Q_2(\sqrt{-1})^{\times})$.
Hence $M_2$ must satisfy \cite[1.2(3)]{EH}. This first gives $r\ge2$. Also, when this occurs we must have
$$\theta(O^+(M_2))=\{x\in\Q_2^{\times}: (x,-bc)_2=1\}={\rm N}_{\Q_2(\sqrt{-1})/\Q_2}(\Q_2(\sqrt{-1})^{\times}).$$
This shows that $-bc\in-\Q_2^{\times2}$, i.e., $bc\equiv 1\pmod 8$. Moreover, as $M_2$ satisfies \cite[1.2(3)]{EH} we further have $(2^rb(a+2^rc),-bc)_2=1$. This gives $a\equiv b\pmod4$. Hence {\rm (i)} holds. With the same method in the proof of Theorem \ref{Thm. A}, the condition {\rm(iv)} holds.

This completes the proof.\qed

Before the proof of Theorem \ref{Thm. D}. We briefly discuss the Jordan decomposition of $M_2$ in the case $r=s=0$. By the symmetry of $a,b,c$ in this case we may assume $b\equiv c\pmod4$. By computation one can verify that
$$M_2\cong\langle\ve\rangle\perp4\ve\begin{pmatrix} a(b+c) & -ab \\-ab & b(a+c)\end{pmatrix}.$$
Let $K$ be the lattice $\begin{pmatrix} a(b+c) & -ab \\-ab & b(a+c)\end{pmatrix}$. Then $K$ is unimodular and ${\bf n}(K)=2\Z_2$ (where ${\bf n}(K)$ denotes the norm of the lattice $K$).
By \cite[Lemma 5.2.6]{Ki} we have
\begin{equation}\label{Eq. M2 in the case r=s=0}
M_2\cong\begin{cases}\langle\ve\rangle\perp4\ve\begin{pmatrix} 2 & 1 \\1 & 2\end{pmatrix}&\mbox{if}\ a\equiv b\equiv c\pmod4,\\\\ \langle\ve\rangle\perp4\ve\begin{pmatrix} 0 & 1 \\1 & 0\end{pmatrix}&\mbox{if}\ a\not\equiv  b\equiv c\pmod4.\end{cases}
\end{equation}
Now we prove our last theorem.

{\bf Proof of Theorem \ref{Thm. D}}. We shall show that there are no primitive spinor exceptions of $\gen(M)$ in $\mathcal{A}$, which implies that $\mathcal{F}$ is almost universal by the discussion below the proof of Lemma \ref{Lemma local obstruction}. Suppose that there is a primitive spinor exception $t\in\mathcal{A}$. We divide the remaining part of proof into the following cases.

{\bf Case I.} $\nu_3(abc)$ is odd.

By Lemma \ref{Lemma quadratic field} we have $\Q(\sqrt{-tdM})=\Q(\sqrt{-3})$. By (\ref{Eq. M2 in the case r=s=0}) we have
$$M_2\cong\begin{cases}\langle\ve\rangle\perp4\ve\begin{pmatrix} 2 & 1 \\1 & 2\end{pmatrix}&\mbox{if}\ a\equiv b\equiv c\pmod4,\\\\ \langle\ve\rangle\perp4\ve\begin{pmatrix} 0 & 1 \\1 & 0\end{pmatrix}&\mbox{if}\ a\not\equiv  b\equiv c\pmod4.\end{cases}$$
By \cite[Lemma 5.2.6]{Ki} we have
\begin{equation}\label{Eq. values of two unimodular lattice with norm (2)}
\q\(\begin{pmatrix} \rho & 1 \\1 & \rho\end{pmatrix}\)=\begin{cases}2\Z_2&\mbox{if}\ \rho=0,\\\{0\}\cup\{x\in\Z_2: \nu_2(x)\equiv1\pmod2\}&\mbox{if}\ \rho=2.\end{cases}
\end{equation}
Noting that $\Q_2(\sqrt{-3})/\Q_2$ is unramified, by \cite[Theorem 2(a)]{EHH} and (\ref{Eq. values of two unimodular lattice with norm (2)})
we have
$$\theta(O^+(M_2))\nsubseteq{\rm N}_{\Q_2(\sqrt{-3})/\Q_2}(\Q_2(\sqrt{-3})^{\times}).$$
This contradicts (\ref{Eq. A}).

{\bf Case II.} $\nu_3(abc)$ is even.

By Lemma \ref{Lemma quadratic field} we have $\Q(\sqrt{-tdM})=\Q(\sqrt{-1})$. Now we consider $M_2$.
By (\ref{Eq. values of two unimodular lattice with norm (2)}) and \cite[1.2]{EH}, we have $\theta(O^+(M_2))=\Q_2^{\times}$ or $\Z_2^{\times}\Q_2^{\times2}$,
which is not contained in ${\rm N}_{\Q_2(\sqrt{-1})/\Q_2}(\Q_2(\sqrt{-1})^{\times})$. This contradicts (\ref{Eq. A}).

In view of the above cases, we see that there are no primitive spinor exceptions in $\gen(M)$. Hence
$\mathcal{F}$ is almost universal.

This completes the proof.\qed

\Ack\ This research was supported by the National Natural Science Foundation of
China (Grant No. 11971222).

\end{document}